\documentclass[12pt,reqno]{amsart}
\usepackage{amsfonts}
\usepackage{amssymb,color}
\usepackage{amssymb}

\usepackage[colorlinks=true]{hyperref}
\hypersetup{urlcolor=blue, citecolor=blue}

\usepackage[margin=2.6cm, a4paper]{geometry}

\newtheorem{theo}{Theorem}[section]
\newtheorem{lemm}[theo]{Lemma}
\newtheorem{defi}[theo]{Definition}

\newtheorem{prop}[theo]{Proposition}
\newtheorem{rema}[theo]{Remark}
\numberwithin{equation}{section}

\newcommand{\bal}{\begin{align}}
\newcommand{\bbal}{\begin{align*}}
\newcommand{\bes}{\begin{split}}
\newcommand{\ees}{\end{split}}
\newcommand{\beq}{\begin{equation}}
\newcommand{\eeq}{\end{equation}}
\newcommand{\bca}{\begin{cases}}
\newcommand{\eca}{\end{cases}}

\newcommand{\pa}{\partial}

\newcommand{\na}{\nabla}
\newcommand{\De}{\dot{\Delta}}
\newcommand{\de}{\delta}

\newcommand{\cd}{\cdot}
\newcommand{\ep}{\varepsilon}
\newcommand{\dd}{\mathrm{d}}
\newcommand{\B}{\dot{B}}

\newcommand{\R}{\mathbb{R}}

\newcommand{\D}{\mathrm{div}}
\newcommand{\PP}{\mathcal{P}}
\newcommand{\Q}{\mathcal{Q}}
\newcommand{\Z}{\mathbb{Z}}

\begin{document}


\title[compressible magnetohydrodynamic system]{On some large global solutions for the compressible magnetohydrodynamic system}

\author[J. Li]{Jinlu Li}
\address{School of Mathematics and Computer Sciences, Gannan Normal University, Ganzhou 341000, China}
\email{lijinlu@gnnu.cn}

\author[Y. Yu]{Yanghai Yu}
\address{School of Mathematics and Statistics, Anhui Normal University, Wuhu, Anhui, 241002, China}
\email{yuyanghai214@sina.com}

\author[W. Zhu]{Weipeng Zhu}
\address{Department of Mathematics, Sun Yat-sen University, Guangzhou, 510275, China}
\email{mathzwp2010@163.com}

\begin{abstract}
In this paper we consider the global well-posedness of compressible magnetohydrodynamic system in $\R^d$ with $d\geq2$, in the framework of the critical Besov spaces. We can show that if the initial data, the shear viscosity and the magnetic diffusion coefficient are small comparing with the volume viscosity, then compressible magnetohydrodynamic system has a unique global solution.
\end{abstract}

\maketitle

\section{Introduction}
The present paper is devoted to the equations of magnetohydrodynamics (MHD) which describe the motion of electrically conducting fluids in the presence of a magnetic field. The barotropic compressible magnetohydrodynamic system can be written as
\beq\label{1.1}\bca
\pa_t\rho+\D(\rho u)=0,\; &\mathrm{in}\quad \R^+\times\R^d\\
\pa_t(\rho u)+\D(\rho u\otimes u)+\na P(\rho)=b\cd \na b-\frac12\na(|b|^2)+\mu \Delta u
\\ \qquad+\na ((\mu+\lambda)\D u),\; &\mathrm{in}\quad \R^+\times\R^d\\
\pa_tb+(\D u)b+u\cd \na b-b\cd \na u-\nu\Delta b=0,\; &\mathrm{in}\quad \R^+\times\R^d\\
\D b=0,\; &\mathrm{in}\quad \R^+\times\R^d
\eca\eeq
where $\rho=\rho(t, x)\in \R^+$ denotes the density, $u=u(t,x)\in\R^d$ and $b=b(t,x)\in\R^d$ stand for the velocity field and the magnetic field, respectively. The barotropic assumption means that the pressure $P=P(\rho)$ is given and assumed to be strictly increasing. The constant $\nu > 0$ is the resistivity acting as the magnetic diffusion coefficient of the magnetic field. The shear and volume viscosity coefficients $\mu$ and $\lambda$ are constant and fulfill the standard strong parabolicity assumption:
\beq
\mu>0, \qquad \kappa=\lambda+2\mu>0.
\eeq
To complete the system (1.1), the initial data are supplemented by
\bal\label{1.2}
(u,b,\rho)(t,x)|_{t=0}=(u_{0}(x),b_{0}(x),\rho_{0}(x))
\end{align}
and also, as the space variable tends to infinity, we assume
\bal\label{1.3}
\lim\limits_{|x|\to\infty}\rho_0(x)=1.
\end{align}
The system of MHD involves various topics such as the evolution and dynamics of astrophysical objects, thermonuclear fusion, metallurgy and semiconductor crystal growth, see for example \cite{Cabannes1970,Landau1960}. Roughly speaking, The system \eqref{1.1} is a coupling between the compressible Navier-Stokes equations with the magnetic equations (heat equations). On the other hand, notice that $b \equiv 0$, system \eqref{1.1} reduces to the usual compressible Navier-Stokes system for baratropic fluids. Due to its physical importance, complexity, rich phenomena and mathematical challenges, there have been huge literatures on the study of the compressible MHD problem \eqref{1.1} by many physicists and mathematicians, see for example, \cite{Cabannes1970,Chen1 2002,Chen2 2003,Ducomet2006,Fan1 2007,Fan2 2008,Fan3 2009,Hu1 2008,Hu2 2010,Kawashima1984,Landau1960,Strohmer1990,Umeda 1984,vol 1972,Wang 2010,Zhang 2010} and the references therein. Now, we briefly recall some results concerned
with the multi-dimensional compressible MHD equations in the absence of vacuum, which are more relatively with our problem.  Kawashima \cite{Kawashima1984} established the local and global well-posedness of the solutions to
the compressible MHD equations as the initial density is strictly positive, see also Vol'pert-Khudiaev \cite{Vol1972} and Strohmer \cite{Strohmer1990} for the local existence results. To catch the scaling invariance property, Danchin first introduce in his series papers \cite{Danchin 2000,Danchin1 2001,Danchin2 2001,Danchin2007,Danchin2016} the ``Critical Besov Spaces" which were inspired by those efforts on the incompressible Navier-Stokes. Recently, Danchin et.al prove that the compressible Navier-Stokes system convergence to the homogeneous incompressible case for the large volume viscosity in \cite{Danchin 2017}. Motivated this, our main goal of the present paper is devoted to extend the compressible Navier-Stokes system to the MHD system. That is, we will prove the global existence of strong solutions to \eqref{1.1} for a class of large initial data. We notice that if $\kappa$ tends to $+\infty$, then velocity field and magnetic field  will satisfy the incompressible MHD system:
\beq\label{1.4}\bca
\pa_tU+U\cdot\na U-\mu \Delta U+\na \Pi-B\cd \na B-\frac12\na(|B|^2)=0,\\
\pa_tB+U\cd \na B-B\cd \na U-\nu \Delta B=0,\\
\D U=\D B=0, \qquad (U,B)|_{t=0}=(U_0,B_0)
\eca\eeq
with $U_0=\PP u_0$ and $B_0=b_0$. Here, the projectors $\PP$ and $\Q$ are defined by
$$\PP:=\mathrm{Id}+(-\Delta)^{-1}\na \D, \qquad \Q:=-(-\Delta)^{-1}\na \D.$$

Our main result can state be stated as follows:
\begin{theo}\label{th1}
Assume that $d\geq 2$, $u_0\in \B^{\frac d2-1}_{2,1}(\R^d)$ and $a_0:=\rho_0-1\in \B^{\frac d2-1}_{2,1}(\R^d)\cap \B^{\frac d2}_{2,1}(\R^d)$. Suppose that \eqref{1.4} generates a unique global solution $(U,B)\in \mathcal{C}(\R^+;\B^{\frac d2-1}_{2,1}(\R^d))$ satisfying $U_0:=\PP u_0$ and $B_0=b_0$. Let $C$ be a large universal constant and denote
\bal\bes\label{1.5}
&M:=||U,B||_{L^\infty(\R^+;\B^{\frac d2-1}_{2,1})}+||U_t,B_t,\mu\na^2 U,\nu\na^2 B||_{L^1(\R^+;\B^{\frac d2-1}_{2,1})},
\\&D_0:=Ce^{C(1+\mu^{-1}+\nu^{-1})(M+1)^2}\big(||a_0,\Q v_0||_{\B^{\frac d2-1}_{2,1}}+\kappa||a_0||_{\B^{\frac d2}_{2,1}}+1\big), \quad
\\&\de_0:=Ce^{2C(1+\mu^{-2}+\nu^{-2})(M+1)^2}\big(\kappa^{-1}D^2_0+\kappa^{-\frac12}D_0\big).
\ees\end{align}
If $\kappa$ is large enough and $||a_0||_{\B^{\frac d2}_{2,1}}$ is small enough such that
$$\kappa^{-1}D_0\ll1, \quad \de_0(\frac1\mu+\frac1\nu+1)\leq \frac12, $$
then \eqref{1.1} has a unique global-in-time solution $(\rho,u,b)$ which satisfies
\bal\label{1.6}\bes
&u,b\in \mathcal{C}(\R^+;\B^{\frac d2-1}_{2,1})\cap L^1(\R^+;\B^{\frac d2+1}_{2,1}),\\&
a:=\rho-1\in C(\R^+;\B^{\frac d2-1}_{2,1}\cap \B^{\frac d2}_{2,1})\cap L^2(\R^+;\B^{\frac d2}_{2,1}).
\ees\end{align}
\end{theo}

\begin{rema}
If $d=2$, according to Lemma \ref{pr-mhd}, we can set
$$M:=C||U_0,B_0||_{\B^0_{2,1}}\exp\Big(C(\frac{1}{\mu^4}+\frac{1}{\nu^4})||U_0,B_0||^4_{L^2}\Big).$$
From Theorem \ref{th1}, we deduce that the system \eqref{1.1} has a unique global-in-time solution without any smallness condition on the initial data. On the other hand, our result improves the the previous one due to Danchin et.al who considered the compressible Navier-Stokes system in \cite{Danchin 2017}.
\end{rema}

\section{Littlewood-Paley analysis}

In this section, we recall the Littlewood-Paley theory, the definition of homogeneous Besov spaces and some useful properties.
First, let us introduce the Littlewood-Paley decomposition. Choose a radial function $\varphi\in \mathcal{S}(\mathbb{R}^d)$ supported in $\tilde{\mathcal{C}}=\{\xi\in\mathbb{R}^d,\frac34\leq \xi\leq \frac83\}$ such that
\begin{align*}
\sum_{j\in \mathbb{Z}}\varphi(2^{-j}\xi)=1 \quad \mathrm{for} \ \mathrm{all} \ \xi\neq0.
\end{align*}
The frequency localization operator $\dot{\Delta}_j$ and $\dot{S}_j$ are defined by
\begin{align*}
\dot{\Delta}_jf=\varphi(2^{-j}D)f=\mathcal{F}^{-1}(\varphi(2^{-j}\cdot)\mathcal{F}f), \quad \dot{S}_jf=\sum_{k\leq j-1}\dot{\Delta}_kf \quad \mathrm{for} \quad j\in\mathbb{Z}.
\end{align*}
With a suitable choice of $\varphi$, one can easily verify that
\begin{align*}
\dot{\Delta}_j\dot{\Delta}_kf=0 \quad \mathrm{if} \quad |j-k|\geq2, \quad \dot{\Delta}_j(\dot{S}_{k-1}f\dot{\Delta}_kf)=0 \quad  \mathrm{if} \quad  |j-k|\geq5.
\end{align*}

Now, we will introduce the definition of the homogeneous Besov space. We denote the space $\mathcal{Z}'(\mathbb{R}^d)$ by the dual space of $\mathcal{Z}(\mathbb{R}^d)=\{f\in \mathcal{S}(\mathbb{R}^d);D^{\alpha}\hat{f}(0)=0;\ \forall \alpha\in \mathbb{N}^d\}$, which can be identified by the quotient space of $\mathcal{S}'(\mathbb{R}^d)/\mathcal{P}$ with the polynomials space $\mathcal{P}$. The formal equality $f=\sum\limits_{j\in \mathbb{Z}}\dot{\Delta}_jf$ holds true for $f\in\mathcal{Z}'(\mathbb{R}^d)$ and is called the homogenous Littlewood-Paley decomposition.

The operators $\dot{\Delta}_j$ help us recall the definition of the homogenous Besov space (see \cite{Bahouri2011})

\begin{defi}
Let $s\in \mathbb{R}$, $1\leq p,r\leq \infty$. The homogeneous Besov space $\B^s_{p,r}$ is defined by
\begin{align*}
\B^s_{p,r}=\{f\in \mathcal{Z}'(\mathbb{R}^d):||f||_{\B^s_{p,r}}<+\infty\},
\end{align*}
where
\begin{align*}
||f||_{\B^s_{p,r}}\triangleq \Big|\Big|(2^{ks}||\dot{\Delta}_k f||_{L^p})_{k}\Big|\Big|_{\ell^r}.
\end{align*}
\end{defi}

\begin{rema}\label{re1}
Let $\mathcal{C}'$ be an annulus and $(u_j)_{j\in \Z}$ be a sequence of functions such that
$$Supp\ \hat{u}_j\subset 2^j\mathcal{C}' \quad and \quad ||(2^{js}||u_j||_{L^p})_{j\in \Z}||_{\ell^r}<\infty.$$
There exists a constant $C_s$ depending on $s$ such that
$$||u||_{\B^s_{p,r}}\leq C_s||(2^{js}||u_j||_{L^p})_{j\in \Z}||_{\ell^r}.$$
\end{rema}

Next, we give the important product acts on homogenous Besov spaces by collecting some useful lemmas from \cite{Bahouri2011}.

\begin{lemm}\label{le1}
Let $s_1,s_2\leq \frac d2$, $s_1+s_2>0$ and $(f,g)\in\B^{s_1}_{2,1}(\R^d)\times\B^{s_2}_{2,1}(\R^d)$. Then we have
\begin{align*}
||fg||_{\B^{s_1+s_2-\frac d2}_{2,1}}&\leq C||f||_{\B^{s_1}_{2,1}}||g||_{\B^{s_2}_{2,1}}.
\end{align*}
\end{lemm}

\begin{lemm}\label{le2}
Assume that $F\in W^{[\sigma]+2,\infty}_{loc}(\R)$ with $F(0)=0$. Then for any $f\in L^\infty(\R^d)\cap \B^s_{2,1}(\R^d)$, we have
$$||F(f)||_{\B^s_{2,1}}\leq C(||f||_{L^\infty})||f||_{\B^s_{2,1}}.$$
\end{lemm}

\begin{lemm}\label{le3} For $(p,r_1,r_2,r)\in[1,\infty]^4$, $s_1\neq s_2$ and $\theta\in(0,1)$, the following interpolation inequality holds
$$\|u\|_{\dot{B}_{p,r}^{\theta s_1+(1-\theta)s_2}}\leq C\|u\|^{\theta}_{\dot{B}_{p,r_1}^{s_1}}\|u\|^{1-\theta}_{\dot{B}_{p,r_2}^{s_2}}.$$
\end{lemm}

\begin{prop}\label{pr-mhd}
Let $U_0,B_0\in \B^{0}_{2,1}(\R^2)$ with $\D U_0=\D B_0=0$. Then there exists a unique solution to \eqref{1.4} such that
$$U,B\in \mathcal{C}(\R^+;\B^0_{2,1}(\R^2))\cap L^1(\R^+;\B^2_{2,1}(\R^2)).$$
Furthermore, there exists some universal constant $C$, one has for all $T\geq 0$,
\bbal
& ||U,B||_{L^\infty_T(\B^0_{2,1})}+||U_t,B_t,\mu \na^2 U,\nu\na^2 B||_{L^1_T(\B^0_{2,1})}\\& \qquad \leq C||U_0,B_0||_{\B^0_{2,1}}\exp\Big(C(\frac{1}{\mu^4}+\frac{1}{\nu^4})||U_0,B_0||^4_{L^2}\Big).
\end{align*}
\end{prop}
\begin{proof}
For any $t\in[0,T]$, the standard energy balance reads:
\bbal
||U(t)||^2_{L^2}+||B(t)||^2_{L^2}+2\mu\int^t_0||\na U||^2_{L^2}+2\nu\int^t_0||\na B||^2_{L^2}=||U_0||^2_{L^2}+||B_0||^2_{L^2},
\end{align*}
which implies for all $T\geq 0$,
\bal\label{2.1}
\mu^{\frac14}||U||_{L^4_T(\B^{\frac12}_{2,1})}+\nu^\frac14||B||_{L^4_T(\B^{\frac12}_{2,1})}\leq C||U_0,B_0||_{L^2}.
\end{align}
From the estimates of the Stokes system in homogeneous Besov spaces, we have
\bal\label{2.2}\bes
&\quad ||U,B||_{L^\infty_T(\B^0_{2,1})}+||U_t,B_t,\mu \na^2 U,\nu \na^2 B||_{L^1_T(\B^0_{2,1})}
\\& \leq C\big(||U_0,B_0||_{\B^0_{2,1}}+||U\cd \na U-B\cd \na B||_{L^1_T(\B^0_{2,1})}+||B\cd \na U-U\cd \na B||_{L^1_T(\B^0_{2,1})}\big).
\ees\end{align}
In view of the interpolation inequality and Young inequality, we deduce that
\bal\label{2.3}\bes
||U\cd\na U||_{L^1_T(\B^0_{2,1})}&\leq C\int^T_0||U||_{\B^\frac12_{2,1}}||\na U||_{\B^\frac12_{2,1}}\dd t
\\&\leq \int^T_0||U||_{\B^\frac12_{2,1}}||\na U||^{\frac14}_{\B^{-1}_{2,1}}||\na U||^{\frac34}_{\B^{1}_{2,1}}\dd t
\\&\leq \frac{C}{\ep^3\mu^3}\int^T_0||U||^{4}_{\B^{\frac12}_{2,1}}||U||_{\B^0_{2,1}}\dd t+\ep \mu||\na^2U||_{L^1_T(\B^0_{2,1})}.
\ees\end{align}
Similar argument as in \eqref{2.3}, we obtain
\bal\bes\label{2.4}
&||B\cd\na B||_{L^1_T(\B^0_{2,1})}\leq \frac{C}{\ep^3\nu^3}\int^T_0||B||^{4}_{\B^{\frac12}_{2,1}}||B||_{\B^0_{2,1}}\dd t+\ep \nu||\na^2 B||_{L^1_T(\B^0_{2,1})},\\&
||U\cd\na B||_{L^1_T(\B^0_{2,1})}\leq \frac{C}{\ep^3\nu^3}\int^T_0||U||^{4}_{\B^{\frac12}_{2,1}}||B||_{\B^0_{2,1}}\dd t+\ep \nu||\na^2 B||_{L^1_T(\B^0_{2,1})},\\&
||B\cd\na U||_{L^1_T(\B^0_{2,1})}
\leq \frac{C}{\ep^3\mu^3}\int^T_0||B||^{4}_{\B^{\frac12}_{2,1}}||U||_{\B^0_{2,1}}\dd t+\ep \mu||\na^2U||_{L^1_T(\B^0_{2,1})}.
\ees\end{align}
Therefore, combing \eqref{2.2}-\eqref{2.4} and choosing $\ep$ small enough, we find that
\bbal\bes
&\quad ||U,B||_{L^\infty_T(\B^0_{2,1})}+||U_t,B_t,\mu \na^2 U,\nu\na^2 B||_{L^1_T(\B^0_{2,1})}\\&\leq C\big((\frac{1}{\mu^3}+\frac{1}{\nu^3})\int^T_0(||U||^{4}_{\B^{\frac12}_{2,1}}+||B||^{4}_{\B^{\frac12}_{2,1}})(||U||_{\B^0_{2,1}}+||B||_{\B^0_{2,1}})\dd t+||U_0,B_0||_{\B^0_{2,1}}\big).
\ees\end{align*}
It follows from the Gronwall inequality and \eqref{2.1} that the desired result of this lemma.
\end{proof}

\section{The proof of the main results}

In this section, we shall give the main details for the proof of Theorem \ref{th1}. Our main idea basically follows from the recent work in \cite{Danchin 2017}

Setting $a=\rho-1$, we infer from \eqref{1.1} that
\beq\label{3.1}\bca
\pa_ta+\D(a u)+\D u=0,\\
\pa_tu+u\cd\na u+P'(1+a)\na a-b\cd \na b+\frac12\na(|b|^2)-\mu \Delta u-\na ((\mu+\lambda)\D u)\\
\qquad =-a(u_t+u\cd\na u),\\
\pa_tb+(\D u)b+u\cd \na b-b\cd \na u-\nu\De b=0,\\
\D b=0.
\eca\eeq
Before continue on, we recall the following local well-posedness of the system \eqref{3.1}.
\begin{theo} \cite{Li 2017}\label{the1.0} Assume that the initial data $(a_0:=\rho-1, u_0,b_0)$ satisfy ${\rm{div}}b_0 = 0$ and
$$(a_0,u_0,b_0)\in\dot{B}_{2,1}^{\frac{d}{2}}\times\dot{B}_{2,1}^{\frac{d}{2}-1}\times\dot{B}_{2,1}^{\frac{d}{2}-1}.$$
In addition, $\inf_{x\in\R^d}a_0(x)>-1$, then there exists some time $T > 0$ such that
the system \eqref{3.1} has a local unique solution $(a, u,b)$ on $[0,T]\times\R^d$ which belongs to the function space
$$E_T:=\widetilde{\mathcal{C}}([0,T];\dot{B}_{2,1}^{\frac{d}{2}})\times(\widetilde{\mathcal{C}}([0,T];\dot{B}_{2,1}^{\frac{d}{2}-1})\cap L_T^1\dot{B}_{2,1}^{\frac{d}{2}+1})^{2d},$$
where $\widetilde{\mathcal{C}}([0,T];\dot{B}_{q,1}^{s}):=\mathcal{C}([0,T];\dot{B}_{q,1}^{s})\cap \widetilde{L}^{\infty}([0,T];\dot{B}_{q,1}^{s})$ with $s\in\R$ and $1\leq q\leq\infty$.
\end{theo}
We set $$v=u-U\quad\mbox{and}\quad c=b-B.$$
From the very beginning, the potential $\Q v$ and divergence-free $\PP v$ parts of $v$ are treated separately. Applying $\mathcal{Q}$ to the velocity equation of \eqref{3.1} and noticing that $\Q v=\Q u$ yield
\bal\label{3.2}
\pa_t(\Q v)+\Q((v+U)\cd \na \Q v)-\kappa\De \Q v+\na a=-\Q(aU_t+av_t)-\Q R_1,
\end{align}
where, denoting $k(a)=P'(1+a)-1$
\bal\begin{split}\label{3.2-1}
R_1&=(1+a)(v+U)\cd \na \PP v+(1+a)(v+U)\cdot \na U+a(v+U)\cd \na \Q v
\\&\quad +k(a)\na a-(B+c)\cd\na (B+c)+\frac12\na (|B+c|^2).
\end{split}\end{align}

In view of the density equation of \eqref{3.1} and using $u=\Q v+\PP v+U$, we find that $a$ satisfies
\bal\label{3.3}
\pa_ta+(v+U)\cd \na a+\D \Q v=-a\D \Q v.
\end{align}

Because $\PP U=U$ and $\PP(\Q v\cd \na \Q v)=\PP(a\na a)=0$, applying $\PP$ to the velocity equation of \eqref{3.1}, we discover that
\bal\label{3.4}
\pa_t(\PP v)+\PP((v+U)\cd \na \PP v)-\mu\De \PP v=-\PP(aU_t+av_t+a\na a)-\PP R_2,
\end{align}
where
\bal\label{3.4-1}\begin{split}
R_2&=(1+a)(v+U)\cd \na \Q v+(1+a)v\cd\na U+a(v+U)\cd \na \PP v
\\& \quad +aU\cd \na U-(B+c)\cd\na c-c\cd\na B
\\&=(1+a)\PP v\cd\na (U+\Q v)+(1+a)U\cd\na \Q v+(1+a)\Q v\cd \na U
\\& \quad +a(v+U)\cd \na \PP v+aU\cd\na U+a\Q v\cd\na \Q v-(B+c)\cd\na c-c\cd\na B.
\end{split}\end{align}

According to the  magnetic equation of \eqref{3.1}, we can show that $c$ satisfies
\bal\label{3.5}
\pa_tc+(v+U)\cd \na c-\nu\De c=-R_3,
\end{align}
where
\bal\label{3.5-1}
R_3=(\D \Q v)B+(\D \Q v)c+v\cd \na B-(B+c)\cd \na v-c\cd \na U.
\end{align}

In the sequel, we denote $a^\ell$ and $a^h$ the low and high frequencies parts of $a$ as
\bbal
a^\ell=\sum_{2^{j}\kappa\leq 1}\De_ja, \qquad a^h=\sum_{2^{j}\kappa> 1}\De_ja,
\end{align*}
and set
\bbal
&X_d(T)=||\Q v,a,\kappa \na a||_{L^\infty_T(\B^{\frac d2-1}_{2,1})}+||\Q v_t+\na a,\kappa \na^2\Q v,\kappa\na^2a^\ell,\na a^h||_{L^1_T(\B^{\frac d2-1}_{2,1})},
\\& Y_d(T)=Y_{d,1}(T)+Y_{d,2}(T):=||\PP v,c||_{L^\infty_T(\B^{\frac d2-1}_{2,1})}+||\PP v_t,c_t,\mu \na^2\PP v,\nu\na^2c||_{L^1_T(\B^{\frac d2-1}_{2,1})},
\\& Z_d(T)=||U,B||_{L^\infty_T(\B^{\frac d2-1}_{2,1})}+||U_t,B_t,\mu\na^2 U,\nu\na^2 B||_{L^1_T(\B^{\frac d2-1}_{2,1})},
\\& X_d(0)=||a_0,\Q v_0||_{\B^{\frac d2-1}_{2,1}}+\kappa||a_0||_{\B^{\frac d2}_{2,1}}.
\end{align*}

It is easy to show that
\bal\label{lll0}
Z_d(T)\leq M  \quad \mathrm{for} \ \mathrm{all} \ T>0.
\end{align}
We concentrate our attention on the proof global in time a priori estimates, as the local well-posedness issue has been ensured by Theorem \ref{the1.0}. We claim that if $\kappa$ is large enough then one may find some (large) $D$ and (small) $\delta$ so that there holds for all $T<T^*$,
\bal\bca\label{lll}
X_d(T)\leq D, \quad Y_d(T)\leq \de, \quad \kappa^{-1}D\ll1,
\\\de(\frac1\mu+\frac1\nu+1)\leq 1, \quad D\geq (M+1), \quad ||a(t,\cdot)||_{L^\infty}\leq \frac12.
\eca\end{align}

\text{\bf Step 1. Estimate on the terms $\PP v$ and $c$.}

We first consider the estimates for $\PP v$. Applying $\De_j$ to \eqref{3.4}, taking the $L^2$ inner product with $\De_j \PP v$ then using that $\PP^2=\PP$, we deduce that
\bal\label{3.9}
\frac12\frac{\dd}{\dd t}||\De_j\PP v||^2_{L^2}&+\mu||\na \De_j\PP v||^2_{L^2}=\int_{\R^d}([v+U,\De_j]\cd \na \PP v)\cd \De_j \PP v\dd x
\\&-\int_{\R^d}\De_j(aU_t+av_t+a\na a+R_2)\cd \De_j\PP v\dd x-\frac12\int_{\R^d}|\De_j\PP v|^2\D v\dd x.
\end{align}
According to the commutator estimates of Lemma 2.100 in \cite{Bahouri2011}, the commutator term may be estimated as follows:
\bal\label{lyz}
2^{j(\frac d2-1)}||[v+U,\De_j]\cd \na \PP v||_{L^2}\leq Cc_j||\na (v+U)||_{\B^{\frac d2}_{2,1}}||\PP v||_{\B^{\frac d2-1}_{2,1}}, \quad\mbox{with}\quad ||c_j||_{l^1}=1.
\end{align}
 Now, multiplying both sides of \eqref{3.9} by $2^{j(\frac d2-1)}$ and summing up over $j\in\mathbb{Z}$, using Lemma \ref{le1} and \eqref{lyz}, we obtain that
\bal\label{l-1}\begin{split}
& \qquad ||\PP v||_{L^\infty_T(\B^{\frac d2-1}_{2,1})}+\mu||\na^2\PP v||_{L^1_T(\B^{\frac d2-1}_{2,1})}\leq C \int^T_0||\na (v+U)||_{\B^{\frac d2}_{2,1}}||\PP v||_{\B^{\frac d2-1}_{2,1}}\dd t
\\& +C||a(U_t+\PP v_t+(\Q v_t+\na a))||_{L^1_T(\B^{\frac d2-1}_{2,1})}+C||R_2||_{L^1_T(\B^{\frac d2-1}_{2,1})}.
\end{split}\end{align}

In order to bound $||\PP v_t||_{L^1_T(\B^{\frac d2-1}_{2,1})}$, we infer from \eqref{3.4} and \eqref{l-1} that
\bal\label{l-2}\begin{split}
& ||\PP v||_{L^\infty_T(\B^{\frac d2-1}_{2,1})}+||\PP v_t,\mu\na^2\PP v||_{L^1_T(\B^{\frac d2-1}_{2,1})}\leq C\int^T_0||\na (v+U)||_{\B^{\frac d2}_{2,1}}||\PP v||_{\B^{\frac d2-1}_{2,1}}\dd t
\\& \quad +C\int^T_0||(v+U)\cd \na \PP v||_{\B^{\frac d2-1}_{2,1}}\dd t+C\int^T_0||a(U_t+\PP v_t+(\Q v_t+\na a))||_{\B^{\frac d2-1}_{2,1}}\dd t
\\& \quad+C\int^T_0||R_2||_{\B^{\frac d2-1}_{2,1}}\dd t.
\end{split}\end{align}

Next, we will estimate the Besov norm of the right-hand side for \eqref{l-2}. For the second term of the right-hand side for \eqref{l-2}, we can infer from Lemma \ref{le1} that

\bal\label{l-3}\bes
&\qquad ||(v+U)\cd \na \PP v||_{L^1_T(\B^{\frac d2-1}_{2,1})}
\\&\leq  C\int^T_0||\PP v||_{\B^{\frac d2-1}_{2,1}}||\PP v||_{\B^{\frac d2+1}_{2,1}}\dd t+C\int^T_0||\Q v+U||_{\B^{\frac d2}_{2,1}}||\PP v||_{\B^{\frac d2}_{2,1}}\dd t
\\&\leq C\int^T_0||\PP v||_{\B^{\frac d2-1}_{2,1}}||\PP v||_{\B^{\frac d2+1}_{2,1}}\dd t
\\&\quad+\frac{C}{\ep\mu}\int^T_0||(\Q v,U)||^2_{\B^{\frac d2}_{2,1}}||\PP v||_{\B^{\frac d2-1}_{2,1}}\dd t+C\mu\ep||\na^2\PP v||_{L^1_T(\B^{\frac d2-1}_{2,1})}
\\&\leq C\int^T_0||\PP v||_{\B^{\frac d2-1}_{2,1}}||\PP v||_{\B^{\frac d2+1}_{2,1}}\dd t+C\ep Y_d(T)+\frac{C}{\ep\mu}\int^T_0||(\Q v,U)||^2_{\B^{\frac d2}_{2,1}}||\PP v||_{\B^{\frac d2-1}_{2,1}}\dd t.
\ees\end{align}

For the third term of the right-hand side for \eqref{l-2}, we can infer from that
\bal\bes
&\qquad \int^T_0||a(U_t+\PP v_t+(\Q v_t+\na a))||_{\B^{\frac d2-1}_{2,1}}\dd t
\\& \leq C||a||_{L^\infty_T(\B^{\frac d2}_{2,1})}(||U_t||_{L^1_T(\B^{\frac d2-1}_{2,1})}+||\PP v_t||_{L^1_T(\B^{\frac d2-1}_{2,1})}+||\Q v_t+\na a||_{L^1_T(\B^{\frac d2-1}_{2,1})})
\\&\leq C\kappa^{-1}X_d(T)\Big(X_d(T)+Y_d(T)+Z_d(T)\Big).
\ees\end{align}

For the last term of the right-hand side for \eqref{l-2}, in view of Lemma \ref{le1}, we can estimate them into the following parts:
\bal\label{l-4}\bes
&\qquad||(1+a)\PP v\cd\na (U+\Q v)||_{L^1_T(\B^{\frac d2-1}_{2,1})}
\\&\leq C\int^T_0(1+||a||_{\B^{\frac d2}_{2,1}})||U+\Q v||_{\B^{\frac d2+1}_{2,1}}||\PP v||_{\B^{\frac d2-1}_{2,1}} \dd t,
\ees\end{align}

\bal\label{l-5}\bes
&\qquad ||(1+a)(\Q v\cd\na U+U\cd \na \Q v)||_{L^1_T(\B^{\frac d2-1}_{2,1})}
\\&\leq C (1+||a||_{L^\infty_T(\B^{\frac d2}_{2,1})})
||\Q v||_{L^2_T(\B^{\frac d2}_{2,1})}||U||_{L^2_T(\B^{\frac d2}_{2,1})}
\\&\leq C(1+\kappa^{-1}X_d(T))\kappa^{-\frac12}\mu^{-\frac12}X_d(T)Z_d(T),
\ees\end{align}

\bal\label{l-6}\bes
& \qquad ||a(v+U)\cd \na \PP v||_{L^1_T(\B^{\frac d2-1}_{2,1})}
\\&\leq C||a||_{L^\infty_T(\B^{\frac d2}_{2,1})}||v+U||_{L^\infty_T(\B^{\frac d2-1}_{2,1})}||\PP v||_{L^1_T(\B^{\frac d2+1}_{2,1})}
\\&\leq C\kappa^{-1}\mu^{-1}X_d(T)Y_d(T)\Big(X_d(T)+Y_d(T)+Z_d(T)\Big),
\ees\end{align}

\bal\label{l-7}\bes
& \qquad ||aU\cd\na U+a\Q v\cd\na\Q v||_{L^1_T(\B^{\frac d2-1}_{2,1})}
\\&\leq C||a||_{L^\infty_T(\B^{\frac d2}_{2,1})}(||U||_{L^\infty_T(\B^{\frac d2-1}_{2,1})}||U||_{L^1_T(\B^{\frac d2+1}_{2,1})}
+||\Q v||_{L^\infty_T(\B^{\frac d2-1}_{2,1})}||\Q v||_{L^1_T(\B^{\frac d2+1}_{2,1})})
\\& \leq C\kappa^{-1}X_d(T)\Big(\mu^{-1}Z^2_d(T)+\kappa^{-1}X^2_d(T)\Big),
\ees\end{align}

\bal\label{l-8}\bes
&\qquad||(B+c)\cd\na c+c\cd\na B||_{L^1_T(\B^{\frac d2-1}_{2,1})}
\\&\leq C\int^T_0||c||_{\B^{\frac d2-1}_{2,1}}||c||_{\B^{\frac d2+1}_{2,1}}\dd t+C\int^T_0||B||_{\B^{\frac d2}_{2,1}}||c||_{\B^{\frac d2}_{2,1}}\dd t.
\\&\leq C\int^T_0||c||_{\B^{\frac d2-1}_{2,1}}||c||_{\B^{\frac d2+1}_{2,1}}\dd t+C\ep\nu ||c||_{L^1_T(\B^{\frac d2+1}_{2,1})}+\frac{C}{\ep\nu}\int^T_0||B||^2_{\B^{\frac d2}_{2,1}}||c||_{\B^{\frac d2-1}_{2,1}}\dd t
\\&\leq C\int^T_0||c||_{\B^{\frac d2-1}_{2,1}}||c||_{\B^{\frac d2+1}_{2,1}}\dd t+C\ep Y_d(T)
+\frac{C}{\ep\nu}\int^T_0||B||^2_{\B^{\frac d2}_{2,1}}||c||_{\B^{\frac d2-1}_{2,1}}\dd t.
\ees\end{align}

Therefore, summing up \eqref{l-2}-\eqref{l-8}, we obtain
\bal\label{y-1}\begin{split}
&\qquad ||\PP v||_{L^\infty_T(\B^{\frac d2-1}_{2,1})}+||\PP v_t,\mu\na^2\PP v||_{L^1_T(\B^{\frac d2-1}_{2,1})}
\\& \leq C\ep Y_d(T)+C\kappa^{-1}X_d(T)\Big(X_d(T)+Y_d(T)+Z_d(T)\Big)
\\& \quad +C(1+\kappa^{-1}X_d(T))\kappa^{-\frac12}\mu^{-\frac12}X_d(T)Z_d(T)+C\kappa^{-1}X_d(T)\Big(\mu^{-1}Z^2_d(T)+\kappa^{-1}X^2_d(T)\Big)
\\&\quad +C\kappa^{-1}\mu^{-1}X_d(T)Y_d(T)\Big(X_d(T)+Y_d(T)+Z_d(T)\Big)
\\&\quad +C \int^T_0\Big((1+||a||_{\B^{\frac d2}_{2,1}})||(\PP v,\Q v,U,c)||_{\B^{\frac d2+1}_{2,1}}+\frac{1}{\ep\mu}||(\Q v,U)||^2_{B^{\frac d2}_{2,1}}+\frac{1}{\ep\nu}||B||^2_{B^{\frac d2}_{2,1}}\Big)Y_{d,1}\dd t.
\end{split}\end{align}

Now, we estimate the term for $c$. Similar argument as in \eqref{l-2} and \eqref{l-3}, we infer from \eqref{3.5} that
\bal\label{l-9}\bes
& \qquad ||c||_{L^\infty_T(\B^{\frac d2-1}_{2,1})}+||(c_t,\nu\na^2c)||_{L^1_T(\B^{\frac d2-1}_{2,1})}
\\&\leq ||b_0-B_0||_{\B^{\frac d2-1}_{2,1}}+C \int^T_0||(\PP v,\Q v,U)||_{\B^{\frac d2+1}_{2,1}}||c||_{\B^{\frac d2-1}_{2,1}}\dd t
\\& \quad +C||(v+U)\cd \na c||_{L^1_T(\B^{\frac d2-1}_{2,1})}+C||R_3||_{L^1_T(\B^{\frac d2-1}_{2,1})}.
\ees\end{align}

For the last two terms of the right-hand side for \eqref{l-9}, according to Lemma \ref{le1}, we can tackle with them as follows:
\bal\label{l-9.5}\bes
&\qquad ||(v+U)\cd \na c||_{L^1_T(\B^{\frac d2-1}_{2,1})}
\\&\leq  C\int^T_0||\PP v||_{\B^{\frac d2-1}_{2,1}}||c||_{\B^{\frac d2+1}_{2,1}}\dd t+C\int^T_0||\Q v+U||_{\B^{\frac d2}_{2,1}}||c||_{\B^{\frac d2}_{2,1}}\dd t
\\&\leq C\int^T_0||\PP v||_{\B^{\frac d2-1}_{2,1}}||c||_{\B^{\frac d2+1}_{2,1}}\dd t
+\frac{C}{\ep\nu}\int^T_0||(\Q v,U)||^2_{\B^{\frac d2}_{2,1}}||c||_{\B^{\frac d2-1}_{2,1}}\dd t+C\nu\ep||c||_{L^1_T(\B^{\frac d2+1}_{2,1})}
\\&\leq C\int^T_0||\PP v||_{\B^{\frac d2-1}_{2,1}}||c||_{\B^{\frac d2+1}_{2,1}}\dd t+C\ep Y_d(T)+\frac{C}{\ep\nu}\int^T_0||(\Q v,U)||^2_{\B^{\frac d2}_{2,1}}||c||_{\B^{\frac d2-1}_{2,1}}\dd t.
\ees\end{align}

\bal\label{l-10}
||(\D \Q v)c-c\cd \na v-c\cd \na U||_{L^1_T(\B^{\frac d2-1}_{2,1})}\leq C\int^T_0||(U,\PP v,\Q v)||_{\B^{\frac d2+1}_{2,1}}||c||_{\B^{\frac d2-1}_{2,1}}\dd t,
\end{align}

\bal\bes\label{l-11}
&\qquad ||(\D \Q v)B+v\cd \na B-B\cd \na v||_{L^1_T(\B^{\frac d2-1}_{2,1})}
\\&\leq C\int^T_0||\Q v||_{B^{\frac d2}_{2,1}}||B||_{B^{\frac d2}_{2,1}}\dd t+C\int^T_0||\PP v||_{B^{\frac d2}_{2,1}}||B||_{B^{\frac d2}_{2,1}}\dd t
\\&\leq C\int^T_0||\Q v||_{B^{\frac d2}_{2,1}}||B||_{B^{\frac d2}_{2,1}}\dd t+\frac{C}{\ep\nu}\int^T_0||\PP v||_{B^{\frac d2-1}_{2,1}}||B||^2_{B^{\frac d2}_{2,1}}\dd t
+C\ep\nu||\PP v||_{L^1_T(\B^{\frac d2+1}_{2,1})}
\\&\leq C\kappa^{-\frac12}\nu^{-\frac12}X_d(T)Z_d(T)+C\ep Y_d(T)+\frac{C}{\ep\nu}\int^T_0||\PP v||_{B^{\frac d2-1}_{2,1}}||B||^2_{B^{\frac d2}_{2,1}}\dd t.
\ees\end{align}
Hence, collecting the estimates \eqref{l-9}-\eqref{l-11}, we get
\bal\label{y-2}\bes
& \qquad ||c||_{L^\infty_T(\B^{\frac d2-1}_{2,1})}+||(c_t,\nu\na^2c)||_{L^1_T(\B^{\frac d2-1}_{2,1})}
\\&\leq ||b_0-B_0||_{\B^{\frac d2-1}_{2,1}}+\frac{C}{\nu}Y^2_d(T)+C\ep Y_d(T) +C\kappa^{-\frac12}\nu^{-\frac12}X_d(T)Z_d(T)
\\& \quad+C \int^T_0\Big(||(\PP v,\Q v,U,c)||_{\B^{\frac d2+1}_{2,1}}+\frac{1}{\ep\nu}||(\Q v,U,B)||^2_{B^{\frac d2}_{2,1}}\Big)Y_{d}\dd t.
\ees\end{align}
Then, combining \eqref{y-1} and \eqref{y-2} and choosing $\ep$ small enough, we can conclude from Gronwall's inequality that
\bal\label{ly1}\bes
Y_d(T)&\leq Ce^{C||\PP v,\Q v,U,c||_{L^1_T(\B^{\frac d2+1}_{2,1})}+(\frac C\mu+\frac C\nu)||(\Q v,U,B)||^2_{L^1_T(\B^{\frac d2}_{2,1})}}\Big\{\kappa^{-1}X_d(T)\Big(X_d(T)+Y_d(T)+Z_d(T)\Big)
\\& \quad +(1+\kappa^{-1}X_d(T))\kappa^{-\frac12}\mu^{-\frac12}X_d(T)Z_d(T)+\kappa^{-1}X_d(T)\Big(\mu^{-1}Z^2_d(T)+\kappa^{-1}X^2_d(T)\Big)
\\&\quad +\kappa^{-1}\mu^{-1}X_d(T)Y_d(T)\Big(X_d(T)+Y_d(T)+Z_d(T)\Big)+\kappa^{-\frac12}\nu^{-\frac12}X_d(T)Z_d(T)\Big\}.
\ees\end{align}

\text{\bf Step 2. Estimate on the terms $\Q v$ and $a$.}

Now, applying $\De_j$ to \eqref{3.1} and \eqref{3.2} yields that
\bal\label{ll-1}\bca
\pa_ta_j+(v+U)\cd \na a_j+\D \Q v_j=g_j,\\
\pa_t\Q v_j+\Q((v+U)\cd \na \Q v_j)-\kappa \De \Q v_j+\na a_j=f_j,
\eca\end{align}
where
\bal\label{ll-2}\bes
&a_j=\De_ja, \quad \Q v_j=\De_j\Q v, \quad g_j=-\De_j(a \D \Q v)-[\De_j,(v+U)]\cd \na a,
\\& f_j=-\De_j\Q(aU_t+av_t)-\De_j\Q R_1-\Q[\De_j,(v+U)]\cd \na \Q v.
\ees\end{align}

We take the $L^2$ inner product for the first equation of \eqref{ll-1} with $a_j$ and the second equation of \eqref{ll-1} with $\Q v_j$ to obtain
\bal\label{ll-3}\bca
\frac12\frac{\dd}{\dd t}||a_j||^2_{L^2}+(a_j,\D \Q v_j)=\frac12(\D v, a^2_j)+(g_j,a_j), \\
\frac12\frac{\dd}{\dd t}||\Q v_j||^2_{L^2}+\kappa||\na \Q v_j||^2_{L^2}-(a_j,\D \Q v_j)=\frac12(\D v,|\Q v_j|^2)+(f_j,\Q v_j),
\eca\end{align}
We next want to estimate for $||\na a_j||^2_{L^2}$. From the first equation of \eqref{ll-1}, we have
\bal\label{ll-4}
\pa_t\na a_j+(v+U)\cd \na \na a_j+\na \D \Q v_j=\na g_j-\na(v+U)\cd \na a_j.
\end{align}
Following \eqref{ll-4} and second equation of \eqref{ll-1} and taking the $L^2$ inner product, we obtain
\bal\label{ll-5}\bca
\frac12\frac{\dd}{\dd t}||\na a_j||^2_{L^2}+((v+U)\cd \na \na a_j, \na a_j)+(\na \D \Q v_j,\na a_j)\\
\qquad =(\na g_j-\na (v+U)\cd \na a_j,\na a_j),\\
\frac{\dd}{\dd t}(\Q v_j,\na a_j)+(v+U,\na (\Q v_j\cd \na a_j))-\kappa(\De \Q v_j, \na a_j)+||\na a_j||^2_{L^2}
\\ \qquad +(\na \D \Q v_j,\Q v_j)=(\na g_j-\na (v+U)\cd \na a_j,\Q v_j)+(f_j,\na a_j).
\eca\end{align}
Noticing that $(\na \D \Q v_j,\na a_j)=(\De \Q v_j, \na a_j)$ and $\De \Q v_j=\na \D \Q v_j$, we get
\bal\label{ll-6}\bes
& \quad \frac12\frac{\dd}{\dd t}(\kappa||\na a_j||^2_{L^2}+2(\Q v_j\cd \na a_j))+(||\na a_j||^2_{L^2}-||\na \Q v_j||^2_{L^2})
\\&=(\frac12\kappa|\na a_j|^2+\Q v_j\cd \na a_j,\D v)+\kappa(\na g_j-\na(v+U)\cd \na a_j,\na a_j)
\\&\quad +(\na g_j-\na (v+U)\cd \na a_j,\Q v_j)+(f_j,\na a_j).
\ees\end{align}
Multiplying \eqref{ll-6} by $\kappa$ and adding up twice \eqref{ll-3} yield
\bal\label{ll-7}\bes
&\frac12\frac{\dd}{\dd t}\mathcal{L}^2_j+\kappa(||\na \Q v_j||^2_{L^2}+||\na a_j||^2_{L^2})
\\&\quad =\int_{\R^d}(2g_ja_j+2f_j\cd \Q v_j+\kappa^2\na g_j\cd \na a_j+\kappa\na g_j\cd\Q v_j+ \kappa f_j\cd\na a_j)\dd x
\\& \qquad +\frac12\int_{\R^d}\mathcal{L}^2_j\D v\dd x-\kappa\int_{\R^d}(\na(v+U)\cd\na a_j)\cd(\kappa\na a_j+\Q v_j)\dd x,
\ees\end{align}
with
\bal\label{ll-8}\bes
\mathcal{L}^2_j&=\int_{\R^d}(2a^2_j+2|\Q v_j|^2+2\kappa\Q v_j\cd \na a_j+|\kappa\na a_j|^2)\dd x
\\&=\int_{\R^d}(2a^2_j+|\Q v_j|^2+|\Q v_j+ \kappa\na a_j|^2)\dd x\approx ||(\Q v_j,a_j,\kappa \na a_j)||^2_{L^2}.
\ees\end{align}
By \eqref{ll-8}, we obtain
\bbal
\kappa(||\na \Q v_j||^2_{L^2}+||\na a_j||^2_{L^2})\geq c\min(\kappa 2^{2j},\kappa^{-1})\mathcal{L}^2_j,
\end{align*}
which along with \eqref{ll-7} yields
\bal\label{ll-8.1}\bes
\frac12\frac{\dd}{\dd t}\mathcal{L}^2_j+c\min(\kappa 2^{2j},\kappa^{-1})\mathcal{L}^2_j\leq& (\frac12||\D v||_{L^\infty}+||\na (v+U)||_{L^\infty})\mathcal{L}^2_j
\\&+C||(g_j,f_j,\kappa \na g_j)||_{L^2}\mathcal{L}_j.
\ees\end{align}
Multiplying both sides of \eqref{ll-8.1} by $2^{j(\frac d2-1)}$ and then summing up over $j\in\mathbb{Z}$, we infer from Remark \ref{re1} that
\bal\label{ll-9}\bes
&||(a,\kappa\na a,\Q v)||_{L^\infty_T(\B^{\frac d2-1}_{2,1})}+||(\kappa \na^2\Q v,\kappa\na^2a^\ell,\na a^h)||_{L^1_T(\B^{\frac d2-1}_{2,1})}
\\&\leq C ||(a,\kappa\na a,\Q v)(0)||_{\B^{\frac d2-1}_{2,1}}+C\int^T_0||(v,U)||_{\B^{\frac d2+1}_{2,1}}||(a,\kappa \na a,\Q v)||_{\B^{\frac d2-1}_{2,1}}\dd t
\\&\quad +C\int^T_0\sum_{j\in \Z}2^{j(\frac d2-1)}||(g_j,f_j,\kappa\na g_j)||_{L^2}\dd t.
\ees\end{align}
Combining the estimates
\bbal
&||a\D \Q v,\kappa\na(a\D \Q v)||_{L^1_T(\B^{\frac d2-1}_{2,1})}\leq C\int^T_0||\D \Q v||_{\B^{\frac d2}_{2,1}}||a,\kappa\na a||_{\B^{\frac d2-1}_{2,1}}\dd t,
\end{align*}
and
\bbal
& \qquad\int^T_0\sum_j2^{j(\frac d2-1)}||[\De_j,(v+U)]\na a,\kappa\na([\De_j,(v+U)]\na a)||_{L^2}\dd t
\\&\leq C\int^T_0||\na (v+U)||_{\B^{\frac d2}_{2,1}}||a,\kappa\na a||_{\B^{\frac d2-1}_{2,1}}\dd t,
\end{align*}
we have
\bal\label{ll-10}
\int^T_0\sum_{j\in \Z}2^{j(\frac d2-1)}||(g_j,\kappa\na g_j)||_{L^2}\dd t\leq C\int^T_0||(v,U)||_{\B^{\frac d2+1}_{2,1}}||(a,\kappa \na a)||_{\B^{\frac d2-1}_{2,1}}\dd t.
\end{align}
Next, we will estimate the last term $\int^T_0\sum_{j\in \Z}2^{j(\frac d2-1)}||f_j||_{L^2}\dd t$. According to Lemmas \ref{le1}-\ref{le2} and the commutator estimates of Lemma 2.100 in \cite{Bahouri2011}, we have
\bal\label{ll-11}\bes
&\qquad ||(1+a)(v+U)\cd \na (\PP v+U)||_{L^1_T(\B^{\frac d2-1}_{2,1})}
\\&\leq C(1+||a||_{L^\infty_T(\B^{\frac d2}_{2,1})})||(\PP v,U)||_{L^\infty_T(\B^{\frac d2-1}_{2,1})} ||(\na \PP v,\na U)||_{L^1_T(\B^{\frac d2}_{2,1})}
\\& \quad +C\int^T_0(1+||a||_{\B^{\frac d2}_{2,1}})||(\na \PP v,\na U)||_{\B^{\frac d2}_{2,1}}||\Q v||_{\B^{\frac d2-1}_{2,1}}\dd t
\\&\leq C\big(1+\kappa^{-1}X_d(T)\big)\mu^{-1}\big(Y^2_d(T)+Z^2_d(T)\big)
\\& \quad +C\int^T_0(1+||a||_{\B^{\frac d2}_{2,1}})||(\na \PP v,\na U)||_{\B^{\frac d2}_{2,1}}||\Q v||_{\B^{\frac d2-1}_{2,1}}\dd t,
\ees\end{align}

\bal\label{ll-12}\bes
||a(v+U)\cd\na \Q v||_{L^1_T(\B^{\frac d2-1}_{2,1})}&\leq C||a||_{L^\infty_T(\B^{\frac d2}_{2,1})}||(v,U)||_{L^\infty_T(\B^{\frac d2-1}_{2,1})}||\na \Q v||_{L^1_T(\B^{\frac d2}_{2,1})}
\\&\leq C\kappa^{-2}X^2_d(T)\big(X_d(T)+Y_d(T)+Z_d(T)\big),
\ees\end{align}

\bal\label{ll-13}\bes
||k(a)\na a||_{L^1_T(\B^{\frac d2-1}_{2,1})}&\leq C\int^T_0||a||^2_{\B^{\frac d2}_{2,1}}\dd t\leq C\int^T_0(||a^\ell||^2_{\B^{\frac d2}_{2,1}}+||a^h||^2_{\B^{\frac d2}_{2,1}})\dd t
\\&\leq C\int^T_0(||a^\ell||_{\B^{\frac d2-1}_{2,1}}||a^\ell||_{\B^{\frac d2+1}_{2,1}}+||a^h||^2_{\B^{\frac d2}_{2,1}})\dd t\leq C\kappa^{-1}X^2_d(T),
\ees\end{align}

\bal\label{ll-14}
\int^T_0\sum_{j\in\Z}2^{j(\frac d2-1)}||[\De_j,v+U]\na \Q v||_{L^2}\dd t\leq C\int^T_0||\na (v+U)||_{\B^{\frac d2}_{2,1}}||\Q v||_{\B^{\frac d2-1}_{2,1}}\dd t.
\end{align}

\bal\label{ll-15}\bes
&\qquad ||aU_t||_{L^1_T(\B^{\frac d2-1}_{2,1})}+||av_t||_{L^1_T(\B^{\frac d2-1}_{2,1})}
\\&\leq C||(U_t,\PP v_t,\Q v_t+\na a)||_{L^1_T(\B^{\frac d2-1}_{2,1})}||a||_{L^\infty_T(\B^{\frac d2}_{2,1})}+C\int^T_0||a||^2_{\B^{\frac d2}_{2,1}}\dd t
\\&\leq C\kappa^{-1}X_d(T)\Big(X_d(T)+Y_d(T)+Z_d(T)\Big),
\ees\end{align}

\bal\label{ll-16}\bes
&\qquad ||\frac12\na (|B+c|^2)-(B+c)\cd\na (B+c)||_{L^1_T(\B^{\frac d2-1}_{2,1})}
\\&\leq C||(B,c)||_{L^\infty_T(\B^{\frac d2-1}_{2,1})}||(B,c)||_{L^1_T(\B^{\frac d2+1}_{2,1})}\leq \frac{C}{\nu}\big(Z^2_d(T)+ Y^2_d(T)\big).
\ees\end{align}
Therefore, collecting \eqref{ll-9}-\eqref{ll-16},  we have
\bal\label{ly2}\bes
X_d(T)&\leq Ce^{C(1+||a||_{L^\infty_T(\B^{\frac d2}_{2,1})})||\PP v,\Q v,U||_{L^1_T(\B^{\frac d2+1}_{2,1})}}\Big\{X_d(0)
\\& \quad +C\big(1+\kappa^{-1}X_d(T)\big)(\mu^{-1}+\nu^{-1})\big(Y^2_d(T)+Z^2_d(T)\big)
\\& \quad +C\big(\kappa^{-2}X^2_d(T)+\kappa^{-1}X_d(T)\big)\big(X_d(T)+Y_d(T)+Z_d(T)\big)\Big\}.
\ees\end{align}
By \eqref{lll0} and \eqref{lll}, we can deduce that
\bal\bes\label{lll-1}
(1+||a||_{L^\infty_T(\B^{\frac d2}_{2,1})})||\PP v,\Q v,U,c||_{L^1_T(\B^{\frac d2+1}_{2,1})}&\leq (1+\kappa^{-1}D)(\kappa^{-1}D+\mu^{-1}M+\mu^{-1}\de+\nu^{-1}\de)
\\&\leq 2(1+\mu^{-1}+\nu^{-1})(M+1),
\ees\end{align}
and
\bal\bes\label{lll-2}
(\mu^{-1}+\nu^{-1})||(\Q v,U,B)||^2_{L^1_T(\B^{\frac d2}_{2,1})}&\leq (\mu^{-1}+\nu^{-1})\big(\kappa^{-1}D^2+(\mu^{-1}+\nu^{-1})M^2\big)
\\&\leq (1+\mu^{-2}+\nu^{-2})(M+1)^2.
\ees\end{align}
According to \eqref{lll}, \eqref{ly1} and \eqref{ly2}--\eqref{lll-2}, we have
\bal\label{ll-17}\bes
Y_d(T)&\leq Ce^{C(1+\mu^{-2}+\nu^{-2})(M+1)^2}\big(
\kappa^{-1}D^2+\kappa^{-\frac12}(\mu^{-\frac12}+\nu^{-\frac12})DM
 +\kappa^{-1}\mu^{-1}DM^2
\\&\quad +\kappa^{-1}\mu^{-1}D^2\de\big)
\\&\leq Ce^{C(1+\mu^{-2}+\nu^{-2})(M+1)^2}\big(\kappa^{-1}D^2+\kappa^{-\frac12}D\big),
\ees\end{align}
and
\bal\label{ll-18}\bes
X_d(T)&\leq Ce^{C(1+\mu^{-1}+\nu^{-1})(M+1)}\big(X_d(0)+(\mu^{-1}+\nu^{-1})(1+M^2)+M+1\big)
\\&\leq Ce^{C(1+\mu^{-1}+\nu^{-1})(M+1)^2}\big(X_d(0)+1\big),
\ees\end{align}
for a suitable large (universal) constant $C$. So it is natural to take first
\bal\label{ll-19}
D:=Ce^{C(1+\mu^{-2}+\nu^{-2})(M+1)^2}\big(X_d(0)+1\big),
\end{align}
and then to set
\bal\label{ll-20}
\de=Ce^{2C(1+\mu^{-2}+\nu^{-2})(M+1)^2}\big(\kappa^{-1}D^2+\kappa^{-\frac12}D\big).
\end{align}
for a suitable large (universal) constant $C$. It is easy to prove that $||a(t,\cdot)||_{L^\infty}\leq C\kappa^{-1}D$. Therefore, if we make the assumption that $\kappa$ is large enough such that
$$\kappa^{-1}D\ll1, \quad \de(\frac1\mu+\frac1\nu+1)\leq \frac12, $$
then we deduce from \eqref{ll-17}-\eqref{ll-20} that the desired result \eqref{lll}.

{\bf Proof of Theorem 1.1}\quad  First, Theorem \ref{the1.0} implies that there exists a unique
maximal solution $(a,u,b)$ to \eqref{3.1} which belongs to $\widetilde{\mathcal{C}}([0,T];\dot{B}_{2,1}^{\frac{d}{2}})\times(\widetilde{\mathcal{C}}([0,T];\dot{B}_{2,1}^{\frac{d}{2}-1})\cap L_T^1\dot{B}_{2,1}^{\frac{d}{2}+1})^{2d}$ on some time interval $[0, T^*)$, with the global a priori estimates \eqref{lll0} and \eqref{lll} at our hand, then one conclude that $T^*=+\infty$. In fact, let us assume (by contradiction) that $T^*<\infty$. Next, applying \eqref{lll0} and \eqref{lll} for all $t < T^*$ yields
\begin{eqnarray}\label{Equ4.18}
||a,u,b||_{L^\infty_{T^*}(\B^{\frac d2-1}_{2,1})}&\leq&C<\infty.
\end{eqnarray}
Then, for all $t_0 \in[0, T^*)$, one can solve \eqref{3.1} starting with data $(a_0,u_0,b_0)$ at time $t = t_0$ and get a solution according to Theorem \ref{the1.0} on the interval $[t_0, T + t_0]$ with $T$ independent of $t_0$. Choosing $t_0 > T^*- T$ thus shows that the solution can be continued beyond $T^*$, a contradiction.

\vspace*{1em}
\noindent\textbf{Acknowledgements.} This work was partially supported by NSFC (No. 11361004).

\end{document}